\documentclass[12pt]{article}
\usepackage{a4wide}
\usepackage{amsmath,amssymb,amsthm}
\usepackage[mathscr]{eucal}
\usepackage{tikz}
\usetikzlibrary{positioning,automata}
\usetikzlibrary{arrows,calc}
\usetikzlibrary{decorations.markings,plotmarks}
\usepackage{graphics}
\usepackage{color}
\usepackage{leftidx}
\usepackage{hyperref}

\definecolor{red}{rgb}{1,0.00,0.00}

\author{Piotr W. Nowak, Andriy Oliynyk and Veronika Prokhorchuk}

\title{\textbf{On (bi)reversible automata generating lamplighter groups}}

\date{}

\newtheorem{theorem}{Theorem}
\newtheorem{proposition}[theorem]{Proposition}
\newtheorem{corollary}[theorem]{Corollary}
\newtheorem{lemma}{Lemma}
\theoremstyle{definition}

\newtheorem{question}{Question}

\begin{document}

\maketitle
\newcommand{\WangTile}[4]{%
\begin{tikzpicture}[scale=1.3,baseline=-21pt,line width=0.5pt,decoration={
    markings,
    mark=at position 0.5 with {\arrow[xshift=3.333pt]{triangle 60}}},
    ]
 \filldraw[black] (0,0) circle (1pt) (0,-1) circle (1pt) (1,0) circle (1pt) (1,-1) circle (1pt);

 \draw[postaction={decorate}] (0,-1) -- node[left,xshift=-1.5pt,yshift=0.7pt] {$#1$} (0,0);
 \draw[postaction={decorate}] (0,0)  -- node[above,yshift=1.5pt] {$#2$} (1,0);
 \draw[postaction={decorate}] (0,-1) -- node[below,yshift=-1.5pt] {$#3$} (1,-1);
 \draw[postaction={decorate}] (1,-1) -- node[right,xshift=1.5pt,yshift=0.7pt] {$#4$} (1,0);

 \draw (0.2,-0.4) -- (0.4,-0.2);
 \draw (0.2,-0.6) -- (0.6,-0.2);
 \draw (0.2,-0.8) -- (0.8,-0.2);
 \draw (0.4,-0.8) -- (0.8,-0.4);
 \draw (0.6,-0.8) -- (0.8,-0.6);

\end{tikzpicture}
}

\begin{abstract}
    For any nontrivial abelian group $\mathbb{X}$ we construct a reversible (bireversible in case the order of $\mathbb{X}$ is odd) automaton such that its set of states and alphabet are identified with $\mathbb{X}$, transition and output functions are defined via the left and the right regular actions correspondingly and its group splits into the restricted wreath product $\mathbb{X} \wr \mathbb{Z}$, i.e. is a lamplighter group.
\end{abstract}

\section{Introduction}

Finite automata is a powerful tool for defining countable groups with unexpected properties. Their value in mathematics is mainly due to the brilliant example of an infinite 2-group of intermediate growth constructed in \cite{MR0565099}. For an  introduction to the topics related to groups defined by automata we refer to~\cite{MR1841755} and~\cite{MR2162164}. The most notable case is finite state self-similar actions, i.e. the case when a group is generated by all states of a given finite automaton and called the group of this automaton. 

Automata that appeared as underlying tools in the mentioned groups definitions are itself of particular interest. The classes of reversible and bireversible automata were introduced in~\cite{MR1841119}. The natural connection between bireversible automata and square complexes was observed in~\cite{MR2155175} where in particular the first examples of bireversible automata whose groups are free and has Kazhdan`s property ({T}) were constructed. It seems natural to apply bireversible automata to find pure theoretical proofs that certain groups posses property ({T}) (cf.~\cite{MR4224715}). The interplay between groups of bireversible automata and fundamental groups of corresponding square complexes was studied in~\cite{MR4424972}.

Among others groups defined by automata metabelian groups, in particular lamplighter groups, i.e. restricted wreath products of the form $\mathbb{X} \wr \mathbb{Z}$, where the group $\mathbb{X}$ is abelian, form central topic of investigations in many papers of recent years. Initial realizations of lamplighter groups as groups defined by automata can be found in~\cite{MR1866850} and~~\cite{MR1841755}, further examples appeared in~\cite{MR2197829} and~\cite{MR2216703}. Another examples and generalizations were obtained recently in~\cite{MR3765459} and~\cite{MR4077656}.

The first example of a bireversible automaton whose group is a lamplighter group was presented in~\cite{MR3520274}. The proof is purely combinatorial and strongly relies on the property called self-duality of the constructed automaton. A bireversible automaton whose group is a semidirect product of a lamplighter group with the cyclic group of order 2 is presented in~\cite{MR3523126}. The newest examples of bireversible  automata whose groups are lamplighter groups were constructed in~\cite{MR3963088}, \cite{MR4118629} and~\cite{MR4581338}, where detailed overview of already mentioned and other related results can be found. 

Let us briefly describe our results. For arbitrary finite abelian group ${\mathbb{X}}$ we define an oriented square complex $\Delta_{\mathbb{X}}$ such that it defines a reversible automaton $\mathcal{A}_{\mathbb{X}}$. This automaton is bireversible if and only if the order of the group ${\mathbb{X}}$ is odd. Both the set of states of $\mathcal{A}_{\mathbb{X}}$ and its alphabet are the underlying set of the group ${\mathbb{X}}$. Then its transition and output functions are defined by the left and the right regular action of the group ${\mathbb{X}}$ correspondingly. The main result of the paper is the following statement.  For every abelian group ${\mathbb{X}}$ of order $n\ge 2$ the group of the automaton $\mathcal{A}_{\mathbb{X}}$ is isomorphic to the lamplighter group  $\mathbb{X} \wr \mathbb{Z}$. In particular it means that the finite state wreath power (\cite{MR2796035}) of a regular group $\mathbb{X}$  contains  the lamplighter group  $\mathbb{X} \wr \mathbb{Z}$. In case of cyclic groups of prime power order constructed automata are minimal with respect to the size of the alphabet (cf.~\cite{MR4200867}). We also discuss some properties of the fundamental group $\pi_1(\Delta_{\mathbb{X}})$ and formulate generalizations and open questions.

The paper is organized as follows.
We recall briefly in Section~{\ref{section:preliminaries}} basic definitions, notations and properties about automata, automaton permutations and groups defined by automata. For a comprehensive introduction in this topic one can refer to~\cite{MR1841755} and~\cite{MR2162164}. 
We also present here required background on square complexes and their connection with automata. Our presentation is based on \cite{MR2155175} and \cite{MR4424972} where one can find omitted details. In Section~\ref{section:construction_complexes} for arbitrary finite abelian group we define corresponding square complex and permutational automaton and present their basic properties. The main result of the paper is proved in Section~\ref{section:groups_are_lamplighters}. We define some generalizations of the presented construction and formulate a few open questions in Section~\ref{section:generalization_and-questions}.

\section{Preliminaries}
\label{section:preliminaries}

\subsection{Automata and groups defined by automata}

Let $\mathsf{X}$ be a finite alphabet, $n=|\mathsf{X}|$ and $n\ge 2$. Denote by $\mathsf{X}^{*}$ the set of all finite words over $\mathsf{X}$ including the empty word $\Lambda$. Then $\mathsf{X}^{*}$ is a free monoid with basis $\mathsf{X}$ under concatenation.  An automaton over $\mathsf{X}$ is a triple $\mathcal{A} = (Q, \lambda, \mu)$ where $Q$ is the set of states, $\lambda:Q \times \mathsf{X} \to Q$ is the transition function and $\mu:Q \times \mathsf{X} \to \mathsf{X}$  is the output function of the automaton $\mathcal{A}$. The automaton $\mathcal{A}$ is called finite if the set $Q$ of its states is finite.

Functions $\lambda$ and $\mu$ can be recursively extended to the set $Q \times \mathsf{X}^{*}$ by the rules
\[
\lambda(q,\Lambda)=q, \quad \lambda(q,xw)=\lambda(\lambda(q,x),w),
\]
\[
\mu(q,\Lambda)=\Lambda, \quad \mu(q,xw)=\mu(q,x)\mu(\lambda(q,x),w),
\]
where $q\in Q$, $x\in \mathsf{X}$, $w\in \mathsf{X}^*$. Note, that extended output function $\mu$ takes its values in $\mathsf{X}^{*}$. For every state $q \in Q$ the 
restriction of $\mu$ at $q$  defines a mapping
$\mu_q: \mathsf{X}^* \to \mathsf{X}^*$ such that
\[
\mu_q(w)=\mu(q,w), \quad w \in \mathsf{X}^*.
\]
A convenient way to describe automata relies on labelled oriented graphs. Specifically, the automaton $\mathcal{A}$ is defined by an oriented graph with the set of vertices $Q$ and an arrow labelled by $x_1|x_2$, $x_1,x_2 \in \mathsf{X}$ connects $q_1$ with $q_2$, $q_1,q_2 \in Q$, if and only if
\[
\lambda(q_1,x_1)=q_2, \quad \mu(q_1,x_1)=x_2.
\]

Automata $\mathcal{A}_1 = (Q_1, \lambda_1, \mu_1)$ and $\mathcal{A}_2 = (Q_2, \lambda_2, \mu_2)$ over alphabets $\mathsf{X}_1$ and $\mathsf{X}_2$ correspondingly are called isomorphic if there exist bijections
\[
\varphi:Q_1 \to Q_2 \text{ and } 
\psi:\mathsf{X}_1 \to \mathsf{X}_2
\]
such that for arbitrary $q \in Q_1$ and $x \in \mathsf{X}_1$ the following equalities hold
\[
\lambda_2(\varphi(q),\psi(x))=\varphi(\lambda_1(q,x))
\text{ and }
\mu_2(\varphi(q),\psi(x))=\psi(\mu_1(q,x)).
\]

A transformation $f:\mathsf{X}^* \to \mathsf{X}^*$ is called automaton transformation over $\mathsf{X}$ if there exist an automaton over $\mathsf{X}$ and its state such that $f$ coincides with the mapping defined at this state as above. All automaton transformations over $\mathsf{X}$ form a semigroup under superposition denoted by $SA(\mathsf{X})$ or simply $SA_n$. For an automaton $\mathcal{A}$ over $\mathsf{X}$ the semigroup of $\mathcal{A}$ is a subsemigroup of $SA(\mathsf{X})$ generated by all automaton transformations defined at states of $SA(\mathsf{X})$. For a transformation semigroup $(T,\mathsf{X})$ its finite state wreath power is defined as the subsemigroup of $SA(\mathsf{X})$ consisting of all automaton transformations over $\mathsf{X}$ defined at states of finite automata over $\mathsf{X}$ such that their output functions define at their states transformations from $T$ only.

A one-to-one automaton transformation over $\mathsf{X}$ is called automaton permutation over $\mathsf{X}$. All automaton permutations over $\mathsf{X}$ form a subgroup in $SA(\mathsf{X})$ denoted by $GA(\mathsf{X})$ or simply $GA_n$.
Each automaton permutation $g \in GA(\mathsf{X})$ is defined by a so-called permutational automaton $\mathcal{A}$ at some state $q$. An automaton is called permutational if at every its state the output function defines a permutation on the alphabet. A subgroup of $GA(\mathsf{X})$ generated by all automaton permutations defined at states of a permutational automaton $\mathcal{A}$ is called the group of the automaton $\mathcal{A}$. 

For a permutational automaton $\mathcal{A}=(Q, \lambda, \mu)$ its inverse is the automaton $\mathcal{A}^{-1}=(Q, \bar{\lambda}, \bar{\mu})$ such that for arbitrary 
$q \in Q$, $x,y \in \mathsf{X}$ one has
\[
\bar{\lambda}(q,x)=\lambda(q,x), \text{ and } \bar{\mu}(q,x)=y \text{ if and only if }
\mu(q,y)=x.
\]
If an automaton permutation $g \in GA(\mathsf{X})$ is defined by a permutational automaton $\mathcal{A}$ at its state $q$ then the inverse automaton permutation is defined by the inverse automaton $\mathcal{A}^{-1}$ at the same state $q$.

Let $\mathcal{A}=(Q, \lambda, \mu)$ be a finite automaton over $\mathsf{X}$.
The dual automaton is the automaton $\partial \mathcal{A}=(\mathsf{X}, \partial\lambda, \partial\mu)$ over alphabet $Q$ such that its set of states is $\mathsf{X}$, its output and transition functions $\partial\lambda, \partial\mu$ are defines by the rules
\[
\partial\lambda(x,q)=\mu(q,x), \quad 
\partial\mu(x,q)=\lambda(q,x), \qquad x \in \mathsf{X}, q\in Q.
\]
The dual automaton is well-defined and it ``flips'' the set of states with alphabet and transition function with output function correspondingly. 

A permutational automaton $\mathcal{A}$ is called reversible if its dual automaton $\partial \mathcal{A}$ is permutational as well. A reversible automaton $\mathcal{A}$ is called bireversible if the dual to its inverse $\partial \mathcal{A}^{-1}$ is permutational.

Let automaton permutation $g \in GA(\mathsf{X})$ is defined by a  permutational automaton $\mathcal{A}$ at its state $q$. Denote by $\sigma_g$ the permutation on $\mathsf{X}$, defined as the restriction of the output function at state $q$. The permutation $\sigma_g$ is called the rooted permutation of $g$. For every $x \in \mathsf{X}$ and $w \in \mathsf{X}^{*}$
the action of $g$ on $xw$ has the form
\[
(xw)^{g}=x^{\sigma_g}w^{g_x}
\]
for some automaton permutation $g_x$ called the section of $g$ at $x$. For arbitrary $u \in \mathsf{X}^{*}$ the section $g_{xu}$ of $g$ at $xu$ is defined recursively as
$(g_x)_u$, i.e. the free monoid $\mathsf{X}^{*}$ acts on $GA(\mathsf{X})$ by taking sections. Then $g$ is uniquely determined by its rooted permutation and sections at elements from $\mathsf{X}$ and admits a notation
\[
g=(g_x, x \in \mathsf{X})\sigma_g
\]
called the wreath recursion of $g$.

Let $g=(g_x, x \in \mathsf{X})\sigma_g$, $h=(h_x, x \in \mathsf{X})\sigma_h$ be wreath recursions of automaton permutations from $GA(\mathsf{X})$. Then their multiplication rule corresponding to the right action on $\mathsf{X}^{*}$ has the form
\begin{equation}
    \label{multiplication_rule}
    g\cdot h = (g_x \cdot h_{\sigma_g(x)}, x \in \mathsf{X}) \sigma_g \sigma_h.
\end{equation}
The rule to compute the inverse is
\begin{equation}
    \label{inverse_rule}
    g^{-1} = (g^{-1}_{\sigma^{-1}_g(x)}, x \in \mathsf{X}) \sigma_g^{-1}.
\end{equation}

\subsection{Square complexes and (bi)reversible automata}

A square complex $\Delta$ is a combinatorial 2-complex such that all its 2-cells correspond to paths of length 4, i.e. they are ``squares''. The square complex $\Delta$ is called $\mathcal{VH}$-complex if the set of its 1-cells can be partitioned into 2 parts, the set $V$ of vertical edges and the set $H$ of horizontal edges, such that in each path corresponding to some 2-cell edges from $V$ and $H$ alternate. If an orientation on 1-cells is given and this orientation is preserved under attaching 2-cells then $\Delta$ is called oriented. For an oriented 1-cell $e$ we denote by $\bar{e}$ the edge oriented in opposite direction.  For a 0-cell $v$ of $\Delta$ its link is defined as a graph $Link(v)$ such that its vertex set consists of (oriented) 1-cells of $\Delta$ incident to $v$. Two such vertices are connected in $Link(v)$ if and only if there exists a 2-cell of $\Delta$ in which they have common ``corner'' $v$. For a $\mathcal{VH}$-complex $\Delta$ the partition of its 1-cells induces natural partition on the vertex set of $Link(v)$. Hence, in this case the graph $Link(v)$ is bipartite. If for a square complex $\Delta$ the graph $Link(v)$ is a complete bipartite for each 0-cell $v$ then $\Delta$ is called complete.

To every finite automaton $\mathcal{A}=(Q, \lambda, \mu)$ over alphabet $\mathsf{X}$ an oriented $\mathcal{VH}$-complex $\Delta_{\mathcal{A}}$ is associated. It has exactly one 0-cell, an oriented vertical loop for every $q\in Q$ and an oriented horizontal loop for every $x \in \mathsf{X}$. For every pair $(q,x) \in Q \times \mathsf{X}$ it has  one 2-cell corresponding to the oriented path $(q,x, \overline{\lambda(q,x)}, \overline{\mu(q,x)})$. Then $\Delta_{\mathcal{A}}$ consists of $2(|Q|+|\mathsf{X}|)$ oriented 1-cells and $|Q|\cdot |\mathsf{X}|$ 2-cells.
The automaton $\mathcal{A}$ is bireversible if only if corresponding oriented $\mathcal{VH}$-complex $\Delta_{\mathcal{A}}$ is complete.
The fundamental group $\pi_1(\Delta_{\mathcal{A}})$ of the complex $\Delta_{\mathcal{A}}$ is finitely presented. Its set of generators is $Q \cup \mathsf{X}$ and every pair $(q,x) \in Q \times \mathsf{X}$ gives rise to a defining relation $qx=\mu(q,x)\lambda(q,x)$. In particular, subgroups generated by $Q$ and $\mathsf{X}$ are free. 

Conversely, let $\Delta$ be an oriented $\mathcal{VH}$-complex having exactly one vertex with the set $V$ of vertical edges and the set $H$ of horizontal edges  such that every pair $(v,h)\in V\times H$ of positively oriented edges
starting at the same vertex are incident to a unique square.
Then $\Delta$ defines an automaton 
$\mathcal{A}_{\Delta}=(V,\lambda_{\Delta},\mu_{\Delta})$ over alphabet $H$ such that 
for arbitrary $(v,h)\in V\times H$ one has
\[
\lambda_{\Delta}(v,h)=v_1, \mu_{\Delta}(v,h)=h_1,
\]
where $(v,h, \bar{v_1},\bar{h_1}) $ is the unique square to which $(v,h)$ is incident.
Then the equality ${\Delta}_{\mathcal{A}_{\Delta}}=\Delta $ holds. In particular, the automaton $\mathcal{A}_{\Delta}$ is bireversible if and only if the complex $\Delta$ is complete.

\section{Family of complexes and (bi)reversible automata}
\label{section:construction_complexes}

Fix an additively written finite abelian group $\mathbb{X}$ of order $n\ge 2$. Define an oriented square $\mathcal{VH}$-complex $\Delta_{\mathbb{X}}$ with exactly one $0$-cell. Let its set of oriented horizontal 1-cells be the group $\mathbb{X}$. The set of its oriented vertical 1-cells is $Q_{\mathbb{X}}=\{a_i, i \in \mathbb{X} \}$. Note, that the set $Q_{\mathbb{X}}$ is naturally identified with $\mathbb{X}$. However, for technical convenience we will distinguish them. The set of 2-cell of $\Delta_{\mathbb{X}}$ is defined as $(a_i, j , \overline{a_{i-j}}, \overline{i+j})$, $i,j \in \mathbb{X}$ (see Fig.~\ref{fig:csc_delta_n}).

\begin{figure}[ht]
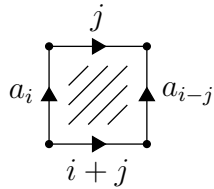

\centering
\WangTile{a_i}{j}{i+j}{a_{i-j}}
\caption{A square in complete square complex $\Delta_{\mathbb{X}}$}
\label{fig:csc_delta_n}
\end{figure}

\begin{lemma}
\label{completness_criterion}
The square complex $\Delta_{\mathbb{X}}$ is complete if and only if the order $n$ of the abelian group ${\mathbb{X}}$ is odd.
\end{lemma}

\begin{proof}
We need to show that each pair of adjacent edges of a 2-cell uniquely determines the cell.

It is easy to see that for arbitrary $\mathbb{X}$ the sets 
\[
 \{ (i,i+j) : i,j \in \mathbb{X} \} \text{ and }
\{ (j,i-j) : i,j \in \mathbb{X} \}
\]
are both equal to $\mathbb{X} \times \mathbb{X}$. 

Moreover, the set 
\[
\{ (i-j,i+j) : i,j \in \mathbb{X} \} 
\]
equals $\mathbb{X} \times \mathbb{X}$ if and only if for arbitrary
$a,b \in \mathbb{X}$ the system of linear equations 
\[
x-y = a, x+y =b
\]
has a solution over $\mathbb{X}$. It is equivalent to 
\[
x+x = a+b, y+y = b-a.
\]
Hence, completness of $\Delta_{\mathbb{X}}$ is equivalent to existence of square roots for all elements in ${\mathbb{X}}$. Since ${\mathbb{X}}$ splits into the direct sum of finite cyclic groups it is sufficient to consider the case of cyclic ${\mathbb{X}}$. It is straightforward, that in such a group square roots exist for all elements if and only if the order $n$ is odd. The statement follows.
\end{proof}

The structure of the fundamental group $\pi_1(\Delta_{\mathbb{X}})$ of the complex $\Delta_{\mathbb{X}}$ can be described as follows.

\begin{proposition}
    \label{fundamental_group_is_hnn}
    The group $\pi_1(\Delta_{\mathbb{X}})$ is an HNN-extension of the free group of rank $n$ associating free subgroups of rank $n^2-n+1$ having index $n$.    
\end{proposition}

\begin{proof}
    Directly follows from \cite[Proposition 1.5]{MR2715704}.
\end{proof}

Denote by $\mathcal{A}_{\mathbb{X}}$ the automaton $(Q_{\mathbb{X}},\lambda_{\mathbb{X}},\mu_{\mathbb{X}})$ over ${\mathbb{X}}$ defined by square complex $\Delta_{\mathbb{X}}$.
Then its transition and output functions $\lambda_{\mathbb{X}}$ and $\mu_{\mathbb{X}}$ are defined by equalities
\[
\lambda_{\mathbb{X}}(a_i,j)=a_{i-j}, \quad \mu_{\mathbb{X}}(a_i,j)=i+j, \qquad i,j \in \mathbb{X}.
\]

\begin{lemma}
\label{self_duality}
The automaton  $\mathcal{A}_{\mathbb{X}}$ is isomorphic to  its dual.
\end{lemma}

\begin{proof}
Consider bijections $\varphi : Q_{\mathbb{X}} \to \mathbb{X}$  and
$\psi : \mathbb{X} \to Q_{\mathbb{X}}$ defined as
\[
\varphi(a_i) = -i, 
\quad
\psi(i) = a_i,
\qquad i \in \mathbb{X}.
\]
Then for transition and output functions $\partial \lambda_{\mathbb{X}}$ and $ \partial \mu_{\mathbb{X}}$ of the dual automaton $\partial \mathcal{A}_{\mathbb{X}}$ the following equalities hold:
\[
\partial \lambda_{\mathbb{X}} (\varphi(a_i), \psi(j))=\partial \lambda_{\mathbb{X}} (-i, a_j)= \mu_{\mathbb{X}}(a_j,-i)=j-i = \varphi(a_{i-j})=\varphi(\lambda_{\mathbb{X}}(a_i,j)),
\]
\[
\partial \mu_{\mathbb{X}} (\varphi(a_i), \psi(j))=\partial \mu_{\mathbb{X}} (-i, a_j)= 
\lambda_{\mathbb{X}}(a_j,-i)=a_{i+j} = \psi(i+j) = \psi(\mu_{\mathbb{X}}(a_{i},j)),
\]
$i,j \in \mathbb{X}$. Therefore, automata $\mathcal{A}_{\mathbb{X}}$ and $\partial \mathcal{A}_{\mathbb{X}}$ are isomorphic.
\end{proof}

Applying Lemma~\ref{completness_criterion} one immediately obtains bireversibility of the automaton $\mathcal{A}_{\mathbb{X}}$ for abelian groups ${\mathbb{X}}$ of  odd order $n>2$.
For abelian groups of even order the automaton $\mathcal{A}_{\mathbb{X}}$ is reversible but not bireversible.

For the cyclic group ${\mathbb{Z}_3}$ the set of 2-cells of the oriented square complex $\Delta_{{\mathbb{Z}_3}}$ is shown on Figure~\ref{fig:3automaton}. 

\begin{figure}[ht]
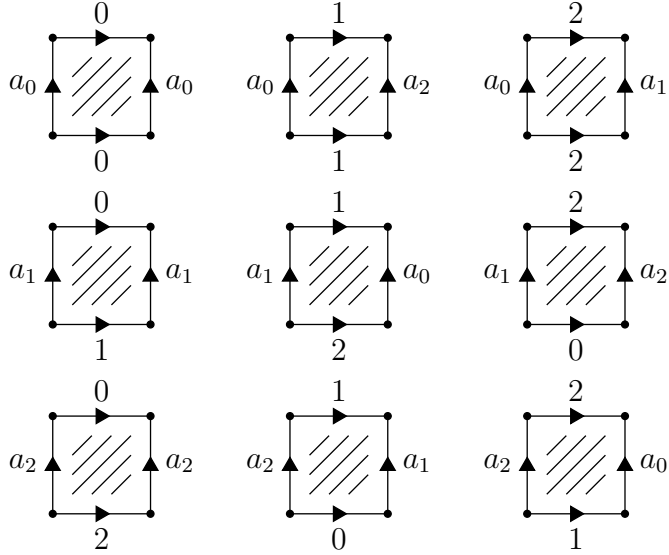

\centering
\[\arraycolsep=1.4pt\def\arraystretch{2.2}
\begin{array}{c}
\WangTile{a_0}{0}{0}{a_0}\quad  \WangTile{a_0}{1}{1}{a_2}\quad  \WangTile{a_0}{2}{2}{a_1}\qquad\\
\WangTile{a_1}{0}{1}{a_1}\quad  \WangTile{a_1}{1}{2}{a_0}\quad  \WangTile{a_1}{2}{0}{a_2} \qquad\\
\WangTile{a_2}{0}{2}{a_2}\quad  \WangTile{a_2}{1}{0}{a_1}\quad  \WangTile{a_2}{2}{1}{a_0} \qquad
\end{array}
\]  
\caption{Oriented complete square complex $\Delta_{\mathbb{Z}_3}$}
\label{fig:csc}
\end{figure}

The bireversible automaton $ \mathcal{A}_{\mathbb{Z}_3}$ and its dual $\partial \mathcal{A}_{\mathbb{Z}_3}$ are shown on Figure~\ref{fig:3automaton} and Figure~\ref{fig:dual3automaton} correspondingly.

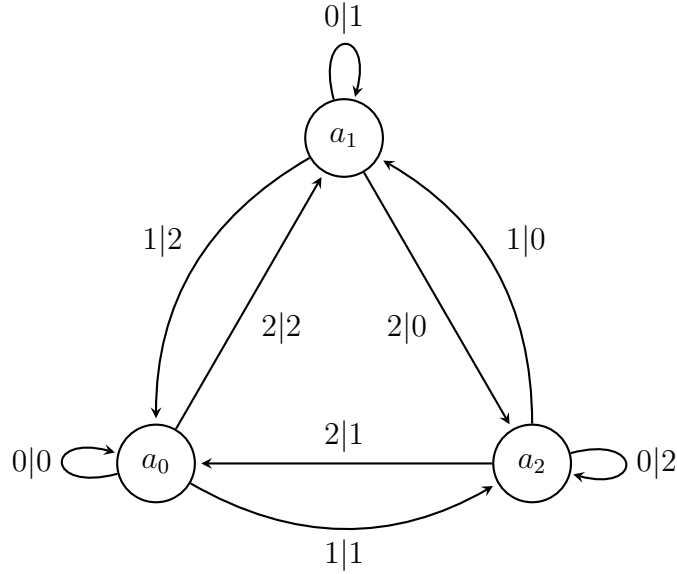
\begin{figure}[ht]
\centering
{\begin{tikzpicture}[>=stealth,scale=0.5,shorten >=2pt, thick,every initial by arrow/.style={*->}]
  \node[state] (c) at (0, 0)   {$a_0$};
  \node[state] (b) at +(0: 10)  {$a_2$};
  \node[state] (a) at +(60: 10) {$a_1$};
    \path[->,scale=4]
    (a) edge  [loop above] node {$0|1$} (a)
    (b) edge  [loop right] node {$0|2$} (b)    
    (c) edge  [loop left] node {$0|0$} (c)
    (a) edge  node [below left]{$2|0$} (b)    
    (b) edge  node [above]{$2|1$} (c)
    (c) edge  node [below right]{$2|2$} (a)
    (b) edge  [bend right] node [above right]{$1|0$} (a)
    (c) edge  [bend right] node [below]{$1|1$} (b)
    (a) edge  [bend right] node [above left]{$1|2$} (c);
\end{tikzpicture}}
\caption{Bireversible automaton $\mathcal{A}_{\mathbb{Z}_3}$}
\label{fig:3automaton}
\end{figure}
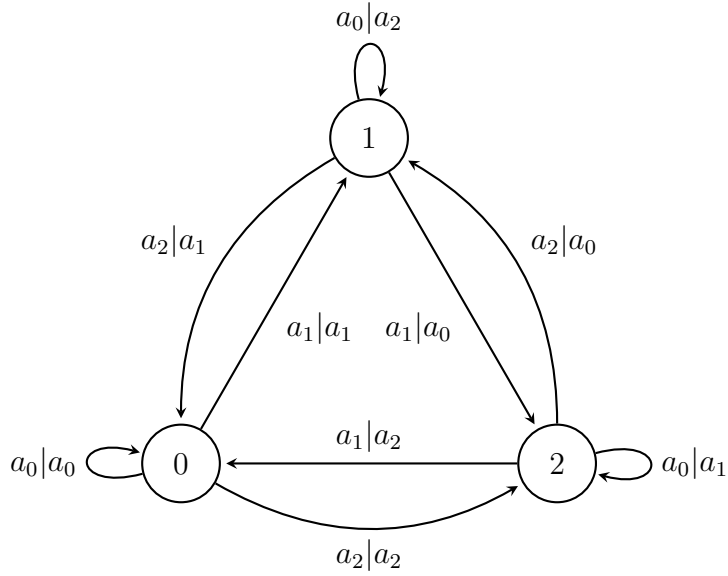
\begin{figure}[ht]
\centering
{\begin{tikzpicture}[>=stealth,scale=0.5,shorten >=2pt, thick,every initial by arrow/.style={*->}]
  \node[state] (c) at (0, 0)   {$0$};
  \node[state] (b) at +(0: 10)  {$2$};
  \node[state] (a) at +(60: 10) {$1$};
    \path[->,scale=4]
    (a) edge  [loop above] node {$a_0|a_2$} (a)
    (b) edge  [loop right] node {$a_0|a_1$} (b)    
    (c) edge  [loop left] node {$a_0|a_0$} (c)
    (a) edge  node [below left]{$a_1|a_0$} (b)    
    (b) edge  node [above]{$a_1|a_2$} (c)
    (c) edge  node [below right]{$a_1|a_1$} (a)
    (b) edge  [bend right] node [above right]{$a_2|a_0$} (a)
    (c) edge  [bend right] node [below]{$a_2|a_2$} (b)
    (a) edge  [bend right] node [above left]{$a_2|a_1$} (c);
\end{tikzpicture}}
\caption{Dual automaton $\partial \mathcal{A}_{\mathbb{Z}_3}$}
\label{fig:dual3automaton}
\end{figure}

\section{Groups of automata $\mathcal{A}_{\mathbb{X}}$}
\label{section:groups_are_lamplighters}

The main result of this section is

\begin{theorem}
\label{automaton_group_is_lamplighter}
    For every abelian group ${\mathbb{X}}$ of order $n\ge 2$ the group of the automaton $\mathcal{A}_{\mathbb{X}}$ is isomorphic to the lamplighter group  
    $\mathbb{X} \wr \mathbb{Z}$.
\end{theorem}

The lamplighter group $\mathbb{X} \wr \mathbb{Z}$, is a semidirect product 
\begin{equation}
\label{semidirect_product_for_lamplighter}
\bigoplus_{i \in \mathbb{Z}} \mathbb{X}^{(i)} \rtimes
\mathbb{Z}, \quad
\mathbb{X}^{(i)} \simeq \mathbb{X}, i \in \mathbb{Z},
\end{equation} 
where $\mathbb{Z}$ acts on the direct sum by translations. Denote by ${\mathbb{G}}_{\mathbb{X}}$ the group of the automaton $\mathcal{A}_{\mathbb{X}}$. Then
\[
{\mathbb{G}}_{\mathbb{X}} = \langle  a_i : i \in {\mathbb{X}}  \rangle.
\]
We will show that ${\mathbb{G}}_{\mathbb{X}}$ splits into a semidirect product as required. The proof follows the general scheme given in \cite{MR3520274} and consists of a few steps.

For every $i \in {\mathbb{X}} $ denote by $\sigma_i$ the permutation on ${\mathbb{X}}$ defined by the rule
\[
x \mapsto x+i, \quad x \in {\mathbb{X}}.
\]
We will use the same notation for its rigid extension, i.e. for the automaton permutation $(e,\ldots,e)\sigma_i$, where $e$ denotes the identity permutation. Then the set $\{\sigma_i : i \in {\mathbb{X}} \}$ form a subgroup isomorphic to ${\mathbb{X}}$.
We denote the identity element of $\mathbb{X}$ by $0$. 

\begin{lemma}
    \label{wreath_recursions}
    Wreath recursions of the generators and their inverses of the group ${\mathbb{G}}_{\mathbb{X}}$ have the form
    \begin{equation}
        \label{generators_wreath_recursion}
        a_i=(a_{i-j}, j \in \mathbb{X})\sigma_{i}, \quad i\in \mathbb{X},
    \end{equation}
    \begin{equation}
        \label{inverses_wreath_recursion}
        a^{-1}_i=(a^{-1}_{-j}, j \in \mathbb{X})\sigma_{-i}, \quad i\in \mathbb{X}.
    \end{equation}
\end{lemma}

\begin{proof}
    Equations~(\ref{generators_wreath_recursion}) for generators of $G_n$ follow from the definition of the oriented square complex $\Delta_n$. It is directly verified using the rule of computing inverses of automaton permutations that the inverse elements of generators of $\mathbb{G}_{\mathbb{X}}$ satisfy 
    equations~(\ref{inverses_wreath_recursion}).    
\end{proof}

Lemma~\ref{wreath_recursions} immediately implies

\begin{lemma}
    \label{sigma_in_gn}
    Automaton permutations $\sigma_i$, $i \in \mathbb{X}$, belong to the group $\mathbb{G}_{\mathbb{X}}$. For arbitrary $i,j \in \mathbb{X}$ the equation 
    \begin{equation}
        \label{a_from_sigma}
        \sigma_i\cdot a_j=a_{i+j}
    \end{equation}
    holds.
\end{lemma}

\begin{proof}
    Equations~(\ref{generators_wreath_recursion}) and~(\ref{inverses_wreath_recursion}) imply 
    \[
    a_0 \cdot a^{-1}_{-i}=
    (a_{-j}, j \in \mathbb{X}) \cdot
    (a^{-1}_{-j}, j \in \mathbb{X})\sigma_{i}
    =\sigma_{i}, \quad i\in \mathbb{Z}_n.
    \]
Hence, $\sigma_i \in \mathbb{G}_{\mathbb{X}}$.

The multiplication rule~(\ref{multiplication_rule}) immediately implies (\ref{a_from_sigma}).
\end{proof}

Let $b_i=(a_{i-j}, j \in \mathbb{X}), i\in \mathbb{X}$. Note, that $b_0=a_0$. Then we have

\begin{lemma}
    \label{b_in_gn}
    Automaton permutation $b_i$, $i\in \mathbb{X}$, belong to the group $\mathbb{G}_{\mathbb{X}}$.
\end{lemma}

\begin{proof}
    Immediately follows from Lemma~\ref{sigma_in_gn} and equalities~(\ref{generators_wreath_recursion}).
\end{proof}

Recall, that for a subgroup $G < GA_n$ the (pointwise) stabilizer of the $m$th level, $m\ge 1$, is the subgroup $St_m(G) <G$ of all automaton permutations from $G$ that fix all words of length m. The group $St_m(G)$ is in fact the intersection of stabilizers of all words of length $m$.

\begin{lemma}
    \label{stabilizers}
    The stabilizer of arbitrary word $w$ over $\mathbb{X}$ in $\mathbb{G}_{\mathbb{X}}$ contains a subgroup such that its action on the set of words with prefix $w$ is isomorphic to the action of $\mathbb{G}_{\mathbb{X}}$ on the set ${\mathbb{X}}^*$.
\end{lemma}

\begin{proof}
    Since the stabilizer of a word is contained in the stabilizer of its prefix it is sufficient to prove the statement for words of length 1 and apply induction by the length of the word.
    
    Let $i \in \mathbb{X}$. Then $b_j \in St_1(\mathbb{G}_{\mathbb{X}})$,  $j \in \mathbb{X}$.   The correspondence 
    \[
    a_j \mapsto b_{i+j}, \quad j \in \mathbb{X},    
    \]
    is an isomorphism between $\mathbb{G}_{\mathbb{X}}$ and the restriction of the subgroup $\langle b_j, j \in  {\mathbb{X}} \rangle $ of $St_1(\mathbb{G}_{\mathbb{X}})$ on the set of words with prefix $i$. Moreover, it agrees with actions on words. The required statement follows.
\end{proof}

The action of a subgroup $G < GA_n$ is called level transitive if for each $m>0$ it is transitive on the set of words of length $m$.

\begin{lemma}
    \label{level_transitivity}
    The group $\mathbb{G}_{\mathbb{X}}$ acts level transitively.
\end{lemma}

\begin{proof}
    Induction by the number $m$ of the level.

    For $m=1$ the statement follows from Lemma~\ref{sigma_in_gn}. Automaton permutations $\sigma_i$, $i \in \mathbb{X}$, form a transitive permutation group on the set of words of length 1.

    Let us prove level transitivity for $m+1$ under the assumption of level transitivity for $m\ge 1$. Consider arbitrary words $u,v$ of length $m+1$ over $\mathbb{X}$. Then $u=u_1i$, $v=v_1j$ for some words $u_1,v_1$  of length $m$ over $\mathbb{X}$ and $i,j \in \mathbb{X}$. By inductive hypotheses there exists an automaton permutation $g \in \mathbb{G}_{\mathbb{X}}$ such that $u_1^g=v_1$. Then  $u^g=(u_1i)^g=v_1k$ for some $k \in \mathbb{X}$. By Lemma~\ref{stabilizers} and the base of induction the stabilizer of $v_1$ contains an automaton permutation $h$ such that $(v_1k)^h=v_1j$. Hence, $u^{gh}=v$. The proof is complete.    
\end{proof}

\begin{lemma}
    \label{level_transitivity_for_dual}
    The group of the dual automaton $\partial A_{\mathbb{X}}$ acts level transitively.
\end{lemma}

\begin{proof}
    Directly follows from Lemma~\ref{self_duality} and Lemma~\ref{level_transitivity}.
\end{proof}

Consider a faithful action of the direct product 
\[
\Pi_{\mathbb{X}}=\prod_{i=1}^{\infty} {\mathbb{X}}^{(i)}, \quad 
{\mathbb{X}}^{(i)} \simeq {\mathbb{X}}, i \ge 1,
\]
by automaton permutations over ${\mathbb{X}}$ defined as follows. For arbitrary word $u=x_1\ldots x_m$ over ${\mathbb{X}}$ and $g=(g_i,i\ge 1) \in \Pi_{\mathbb{X}}$ the image $u^g$ is the word $v=y_1\ldots y_m$ such that  $y_i=x_i+g_i$, $1\le i\le m$. Such an action is defined by a finite automaton if and only if the sequence $g$ is eventually periodic. An automaton permutation $h\in GA_n$ belongs to $\Pi_{\mathbb{X}}$ if and only if
its wreath recursion has the form
\[
h=(h_1, \ldots, h_1) \sigma_{k}
\]
for some $h_1 \in \Pi_{\mathbb{X}}$ and $k\in \mathbb{X}$.

\begin{lemma}
    \label{closed_under_multiplication}
    For arbitrary $h\in \Pi_{\mathbb{X}}$, $i,j \in \mathbb{X}$ both products
    $a_i^{-1} \cdot h \cdot a_j$ and $a_i \cdot h \cdot a_j^{-1}$ belong to $\Pi_{\mathbb{X}}$.
\end{lemma} 

\begin{proof}
Let $h=(h_1, \ldots, h_1) \sigma_{k}$ for some $h_1 \in \Pi_{\mathbb{X}}$, $k\in \mathbb{X}$. Assume that $h_1=(h_2, \ldots, h_2) \sigma_{k_1}$ for some $h_2 \in \Pi_{\mathbb{X}}$ and $k_1 \in \mathbb{X}$. Using Lemma~\ref{wreath_recursions} and multiplication rule~(\ref{multiplication_rule}) for automaton permutations one obtains 
\begin{multline}
\label{eq_closed_under_multiplication}
a_i^{-1} \cdot h \cdot a_j = (a^{-1}_{-l}, l \in \mathbb{X})\sigma_{-i} \cdot 
(h_1, \ldots, h_1) \sigma_{k} \cdot 
(a_{j-l}, l \in \mathbb{X})\sigma_{j}
=\\
(a^{-1}_{-l}h_1, l \in \mathbb{X})\sigma_{k-i} \cdot  
(a_{j-l}, l \in \mathbb{X})\sigma_{j}=
(a^{-1}_{-l}h_1a_{j+k-i-l}, l \in \mathbb{X})\sigma_{j+k-i}.
\end{multline}
To prove that  $a_i^{-1} \cdot h \cdot a_j \in \Pi_{\mathbb{X}}$ we need to show that automaton permutations $a^{-1}_{-l}h_1a_{j+k-i-l}, l \in \mathbb{X}$ are pairwise equal and belong to $\Pi_{\mathbb{X}}$. From  
(\ref{eq_closed_under_multiplication}) 
one obtains
\begin{multline*}
a^{-1}_{-l}h_1a_{j+k-i-l}=
(a^{-1}_{-l_1}h_2a_{j+k-i-l+k_1-(-l)-l_1}, l_1 \in \mathbb{X})\sigma_{j+k-i-l+k_1-(-l)}=\\
(a^{-1}_{-l_1}h_2a_{j+k-i+k_1-l_1}, l_1 \in \mathbb{X})\sigma_{j+k-i+k_1}, \quad l \in \mathbb{X}.
\end{multline*}
Hence, rooted permutations in their wreath recursions are equal.  The sections of the first level have the same form and one can apply induction by the number of the level to obtain the statement required.

Similarly 
\begin{multline}
\label{eq_closed_under_multiplication_1}
a_i \cdot h \cdot a_j^{-1} = (a_{i-l}, l \in \mathbb{X})\sigma_{i} \cdot 
(h_1, \ldots, h_1) \sigma_{k} \cdot 
(a^{-1}_{-l}, l \in \mathbb{X})\sigma_{-j}
=\\
(a_{i-l}h_1, l \in \mathbb{X})\sigma_{i+k} \cdot  
(a^{-1}_{-l}, l \in \mathbb{X})\sigma_{-j}=
(a_{i-l}h_1a^{-1}_{i+k-l}, l \in \mathbb{X})\sigma_{i+k-j}.
\end{multline}

In the same way, to prove that  $a_i \cdot h \cdot a_j^{-1} \in \Pi_{\mathbb{X}}$ we need to show that automaton permutations $a_{i-l}h_1a^{-1}_{i+k-l}, l \in \mathbb{X}$ are pairwise equal and belong to $\Pi_{\mathbb{X}}$. From  
(\ref{eq_closed_under_multiplication_1}) 
one obtains
\begin{multline*}
a_{i-l}h_1a^{-1}_{i+k-l}=
(a_{i-l-l_1}h_2a^{-1}_{i-l+k_1-l_1}, l_1 \in \mathbb{X})\sigma_{i-l+k_1-(i+k-l)}=\\
(a_{i-(l+l_1)}h_2a^{-1}_{i-(l+l_1)+k_1}, l_1 \in \mathbb{X})\sigma_{k_1-k}, l \in \mathbb{X}.
\end{multline*}
As before, rooted permutations in their wreath recursions are equal and sections of the first level have the same form. Applying induction one completes the proof.
\end{proof}

For every $i\in \mathbb{X}$
denote by $c_i$ the automaton permutation defined by the wreath recursion 
\[
c_i=(c_i,\ldots, c_i)\sigma_i.
\]
It is easy to see that the cet  $\{ c_i : i \in \mathbb{X}\} $ form a subgroup in $\Pi_{\mathbb{X}}$ isomorphic to  $\mathbb{X}$.

\begin{lemma}
    \label{a_in_gn}
    The following equations hold:
    \begin{equation}
        \label{eq:a_in_gn}
            a_i^{-1}\cdot  a_j = c_{j-i}, \quad
            i,j \in \mathbb{X}.
    \end{equation}
    In particular, $c_j=a_0^{-1}\cdot a_j$ and $c_j \in {\mathbb{G}}_{\mathbb{X}}$,
    $j \in {\mathbb{X}}$.
\end{lemma}

\begin{proof}
Lemma~\ref{wreath_recursions} and the multiplication rule~(\ref{multiplication_rule}) for automaton permutations imply 
\[
a_i^{-1}\cdot  a_j = 
(a^{-1}_{-l}, l \in \mathbb{X})\sigma_{-i} \cdot
(a_{j-l}, l \in \mathbb{X})\sigma_{j} =
(a^{-1}_{-l}a_{j-i-l}, l \in \mathbb{X})\sigma_{j-i}.
\]
Then for every $l \in \mathbb{X}$ one obtains
\[
a^{-1}_{-l}a_{j-i-l} = 
(a^{-1}_{-l_1}, l_1 \in \mathbb{X})\sigma_{l} \cdot
(a_{j-i-l-l_1}, l_1 \in \mathbb{X})\sigma_{j-i-l} =
(a^{-1}_{-l_i}a_{j-i-l_1}, l_1 \in \mathbb{X})\sigma_{j-i}.
\]
Hence, the sections of the first level have the same rooted permutation. Since  their sections of the first level have the same form the statement follows.
\end{proof}

Denote by $T_{\mathbb{X}}$ the semigroup of the automaton $\mathbb{G}_{\mathbb{X}}$, i.e. the semigroup, generated by the set $\{a_i : i \in \mathbb{X}\}$. 

\begin{lemma}
    \label{sections_of_the_mth_level}
    For arbitrary $u,v\in T_{\mathbb{X}}$ represented as products of $m$ generators
    there exists a section of $u$ of the $m$th level that equals v.
\end{lemma}

\begin{proof}
By Lemma~\ref{level_transitivity_for_dual} the group of the dual automaton $\partial A_{\mathbb{X}}$ acts transitively on the set $\{a_i : i \in \mathbb{X}\}^m$. It means that the semigroup of the dual automaton $\partial A_{\mathbb{X}}$ acts transitively on this set. From the other hand, sections of the $m$th level of an automaton permutation form the set of its images under the action of elements from this semigroup, represented as products of $m$ generators. Now the required statement immediately follows.
\end{proof}

\begin{lemma}
    \label{nontrivial_word}
    Every nontrivial word over the set $\{a_i : i \in \mathbb{X} \}$ defines a nontrivial automaton permutation.
\end{lemma}

\begin{proof}
Let $u$ be a nontrivial word of length $m>1$. By Lemma~\ref{sections_of_the_mth_level} on the $m$th level one of its sections is $a_0^{m-1}a_i$, where $i$ is a nontrivial element of $\mathbb{X}$. Since the rooted permutation of $a_0^{m-1}a_i$ is $\sigma_i$ this section is nontrivial. The proof is complete.
\end{proof}

\begin{lemma}
    \label{sectio_at_00}
    For arbitrary $i,j \in \mathbb{X}$, $k \ge 1$, the section of the product $a_i\cdot a_j$ at the word $0\ldots 0\in \mathbb{X}^k$  is $a_i \cdot a_{j+ki}$.
\end{lemma}

\begin{proof}
Applying the multiplication rule~(\ref{multiplication_rule}) for automaton permutations one obtains
\[
a_i\cdot a_j=
(a_{i-l}, l \in \mathbb{X})\sigma_{i} \cdot 
(a_{j-l}, l \in \mathbb{X})\sigma_{j}=
(a_{i-l}\cdot a_{j+i-l}, l \in \mathbb{X})\sigma_{i+j}.
\]
Hence, the section of the product $a_i\cdot a_i$ at $0$ equals to the product
$a_{i}\cdot a_{j+i}$. Now the required statement immediately follows by induction on $k$.  
\end{proof}

\begin{lemma}
    \label{semigroup_is_free}
    The semigroup $T_{\mathbb{X}}$ is free.
\end{lemma}

\begin{proof}
Assume that $u,v$ are distinct words over $\{a_i : i \in {\mathbb{X}} \}$ that define the same automaton permutation. We will show that this assumption leads to a contradiction.

Consider the case when $u,v$ have different lengths. Without loss of generality we may assume that $v=v_1v_2$ for some nontrivial words $v_1,v_2$ such that $u,v_1$ has the same length $m$. Then $u=a_{i_1}\ldots a_{i_m}$, 
$v_1=a_{j_1}\ldots a_{j_m}$ for some $i_1,\ldots, i_m, j_1,\ldots, j_m \in \mathbb{X}$. It follows from Lemma~\ref{closed_under_multiplication} that the product
\[
a_{j_m}^{-1}\ldots a_{j_1}^{-1}a_{i_1}\ldots a_{i_m} 
\]
belongs to $\Pi_{\mathbb{X}}$. From the other hand this product is equal to the product defined by the word $v_2$. Since the group $\Pi_{\mathbb{X}}$ has the same finite exponent, say $d$, as the group ${\mathbb{X}}$ has, the word $v_2^d$ defines the identity element. It contradicts with Lemma~\ref{nontrivial_word}.

Let now words $u,v$ have the same length $m$ and $m$ is the least length such that there exist distinct words over $\{a_i : i \in {\mathbb{X}} \}$ of this length defining the same automaton permutation. By Lemma~\ref{sections_of_the_mth_level} the automaton permutation defined by $u$ posses a section on the $m$th level defined by the word $a_0^m$. Let the section of the automaton permutation defined by $v$ at the same word of length $m$ over $\mathbb{X}$ is defined by a word $v_1$ of length $m$ over $\{a_i : i\in {\mathbb{X}} \}$, where $v_1=a_{k_1}a_{k_2}\ldots a_{k_m}$, $k_1,k_2,\ldots, k_m \in \mathbb{X}$. Then words $a_0^m$ and $v_1$ of length $m$ define the same automaton permutation $g\in {\mathbb{G}}_{\mathbb{X}}$. The minimality of length $m$ implies $k_1 \ne 0$. 

For arbitrary $k\ge 1$ consider the section $h_k$ of $g$ on $k$th level at the word $0^k\in \mathbb{X}^k$. Using multiplication rule~(\ref{multiplication_rule}) for automaton permutations one obtains that $h_k$ is defined by the product $a_0^m$. From the other hand it is defined by a word $w_k$ of length $m$, computed using $v_1$. Since for every $i\in \mathbb{X}$ the state of $a_i$ at $0^k$ is $a_i$ the first letter of $w_k$ coincides with the first letter of $v_1$. Hence, the automaton permutation $h_k$ is defined by the words $v_1$ and $w_k$ of length $m$ and their first letters are equal. Minimality of the length $m$ now implies the equality of words $v_1$ and $w_k$. 
However, Lemma~\ref{sectio_at_00} implies that the prefix of length 2 of $w_1$ has the form  $a_{k_1}a_{k_2+k_1}$. Since $k_1\ne 0$ it immediately follows that $k_2+k_1\ne k_2$ that leads to a contradiction. The proof is complete.
\end{proof}

Let us proceed with the proof of Theorem~\ref{automaton_group_is_lamplighter}.

\begin{proof}[Proof of Theorem~\ref{automaton_group_is_lamplighter}]
Lemma~\ref{a_in_gn} imply $c_j \in {\mathbb{G}}_{\mathbb{X}}$ and $a_j=c_j\cdot a_0$, $j \in \mathbb{X}$. Hence, the group ${\mathbb{G}}_{\mathbb{X}}$ is generated by $\{c_i : i \in {\mathbb{X}} \}$ and $a_0$.
It follows from Lemma~\ref{nontrivial_word} that $a_0$ has infinite order.

Let $b_{i,j} = a_0^{-i} c_j a_0^{i}$, $i \in \mathbb{Z}$, $j\in {\mathbb{X}}$. 
Denote by $H_{\mathbb{X}}$ the group generated by $\{b_{i,j}, i \in \mathbb{Z}, j \in {\mathbb{X}}\}$.
Since  $a_0^{-1} b_{ij} a_0= b_{i+1,j}$, $i \in \mathbb{Z}$, $j \in {\mathbb{X}}$, the group $H_{\mathbb{X}}$ is a normal subgroup of ${\mathbb{G}}_{\mathbb{X}}$. Moreover, the subgroup, generated by $a_0$, acts on $H_{\mathbb{X}}$ by translations.

Lemma~\ref{closed_under_multiplication} imply that $H_{\mathbb{X}}$ is a subgroup of $\Pi_{\mathbb{X}}$, the countable direct power of the abelian group ${\mathbb{X}}$. Hence, elements $\{b_{i,j}, i \in \mathbb{Z}, j \in {\mathbb{X}}\}$ pairwise commute. Foe every $i \in {\mathbb{Z}}$ the set $\{b_{i,j} : j \in {\mathbb{X}} \}$ form a subgroup isomorphic to ${\mathbb{X}}$. Now in order to prove that the group ${\mathbb{G}}_{\mathbb{X}}$ splits into the semidirect product of the form (\ref{semidirect_product_for_lamplighter}) it is sufficient to show that no product of the form 
\begin{equation}
\label{nontrivial_product}
a_0^{-i_1}c_{j_1} a_0^{i_1} \cdot a_0^{-i_2}c_{j_2} a_0^{i_2} \ldots 
a_0^{-i_m}c_{j_m} a_0^{i_m},
\end{equation}
where integers $i_1,i_2,\ldots, i_m$, $j_1,j_2,\ldots, j_m$, satisfy inequalities $i_1>i_2>\ldots>i_m$, $j_1,j_2,\ldots, j_m \ne 0$, $m\ge 1$,
defines the identity automaton permutation. Applying Lemma~\ref{a_in_gn}
one rewrites
\[
c_{j_1} a_0=a_{j_1}, c_{j_2} a_0=a_{j_2}, \ldots, c_{j_m} a_0=a_{j_m}. 
\]
Hence, (\ref{nontrivial_product})  can be rewritten as
\begin{equation}
\label{nontrivial_product_1}
a_0^{-i_1}a_{j_1} a_0^{i_1-1} \cdot a_0^{-i_2}a_{j_2} a_0^{i_2-1} \ldots 
a_0^{-i_m}a_{j_m} a_0^{i_m-1}
\end{equation}
and further
\begin{equation}
\label{nontrivial_product_2}
a_0^{-i_1}a_{j_1} a_0^{i_1-i_2-1}a_{j_2} a_0^{i_2-i_3-1} \ldots 
a_0^{i_{m-1}-i_m-1}a_{j_m} a_0^{i_m-1}.
\end{equation}
Since $i_1>i_2>\ldots>i_m$ the integers
\[
i_1-i_2-1, i_2-i_3-1, \ldots , i_{m-1}-i_m-1
\]
are non-negative. Then the word
\[
u=a_{j_1} a_0^{i_1-i_2-1}a_{j_2} a_0^{i_2-i_3-1} \ldots a_0^{i_{m-1}-i_m-1}a_{j_m}
\]
is a word over $\{a_i : i \in {\mathbb{X}} \}$ that contain at least once an element $a_j$ with $j\ne 0$. 
Assume that the product (\ref{nontrivial_product}) and therefore the product (\ref{nontrivial_product_2}) is trivial. Then $u$ defines the same automaton permutation as $a_0^{i_1-i_m+1}$. It contradicts with Lemma~\ref{semigroup_is_free}. 
The proof is complete.
\end{proof}

Since at each state of the automaton $\mathcal{A}_{\mathbb{X}}$ the restriction of its output function is a permutation on ${\mathbb{X}}$ defined by the right regular action we immediately obtain

\begin{corollary}
    \label{lamplighter_inside_finite_state_wreath_power}
   For every abelian group ${\mathbb{X}}$ of order $n\ge 2$ the finite state wreath power of ${\mathbb{X}}$ regularly acting on itself contains a subgroup isomorphic to the lamplighter group  $\mathbb{X} \wr \mathbb{Z}$. 
\end{corollary}

\section{Generalizations and open questions}
\label{section:generalization_and-questions}

One of the motivations to consider square complexes in connection with automata and automaton groups is to study their fundamental groups, in particular their residual properties. Since all groups $\mathbb{G}_{\mathbb{X}}$ are infinite the following natural questions arise. 

\begin{question}
    Is it true that for every finite abelian group $\mathbb{X}$ of order $>2$ the fundamental group $\pi_1(\Delta_{\mathbb{X}})$ of the square complex $\Delta_{\mathbb{X}}$ is not residually finite? If not, under what conditions on $\mathbb{X}$ it is not residually finite?
\end{question}

Recall that elements  $g,h$ of a group $G$ generate an anti-torus if for arbitrary non-zero integers $n,m$ the elements $g^{n}$ and $h^{m}$ do not commute. 

Let $\mathbb{X}$ be a finite abelian group of order $>2$.  
\begin{question}
    Is it true that for arbitrary elements $i,j \in \mathbb{X}$ such that at least on of them is nontrivial the subgroup of $\pi_1(\Delta_{\mathbb{X}})$ generated by $a_i$ and $j$ is an anti-torus?  If not, under what conditions on $\mathbb{X}$ and $i,j$ it is an anti-torus?
\end{question}

The square complexes defined in Section~\ref{section:construction_complexes} for finite abelian groups can be defined analogously for arbitrary finite group $\mathbb{Y}$, not necessary abelian. Specifically, let us define an oriented square $\mathcal{VH}$-complex $\Delta_{\mathbb{Y}}$ with exactly one $0$-cell such that the set of its oriented horizontal 1-cells is the group $\mathbb{Y}$ and the set of its oriented vertical 1-cells is $Q_{\mathbb{Y}}=\{q_i, i \in \mathbb{Y} \}$.  The set of 2-cell of $\Delta_{\mathbb{Y}}$ is  $(q_i, j , \overline{a_{j^{-1}i}}, \overline{ij})$, $i,j \in \mathbb{Y}$. One can verify that corresponding automaton $\mathcal{A}_{\mathbb{Y}}$ is reversible and it is bireversible if and only if the group ${\mathbb{Y}}$ has even order.
Denote by $\mathbb{G}_{\mathbb{Y}}$ the group of the automaton $\mathcal{A}_{\mathbb{Y}}$.
\begin{question}
    Is it true that the group $\mathbb{G}_{\mathbb{Y}}$ is isomorphic to the restricted wreath product ${\mathbb{Y}} \wr {\mathbb{Z}}$? If not, is it true that the finite state wreath power of the regular group $\mathbb{Y}$ contains the lamplighter group  $\mathbb{Y} \wr \mathbb{Z}$.
\end{question}

Finally, it would be interesting to examine constructions presented in the paper using methods from~\cite{MR4118629} and~\cite{MR4581338}.
\begin{question}
    Is it possible to obtain automata $\mathcal{A}_{\mathbb{X}}$ applying constructions defined in~\cite{MR4118629} and~\cite{MR4581338}? If not in general, for which finite abelian groups it is possible?
\end{question}

\section*{Acknowledgements}
The research presented in the paper was done during the fellowship of the third author at the Institute of Mathematics of the Polish Academy of Sciences supported by Grant Norweski UMO-2022/01/4/ST1/00026. The first author was supported by Grant MAESTRO-13 UMO-2021/42/A/ST1/00306.

    

\bibliographystyle{plainurl}
\bibliography{references}

\end{document}